\documentclass[12pt]{amsart}

\usepackage{graphicx}
\usepackage{amsmath}
\usepackage{amssymb}
\usepackage{amsfonts}
\usepackage{hyperref}
\usepackage{semantic}

\newcommand{\ipc}{{\sf IPC}}
\newcommand{\iqc}{{\sf IQC}}
\newcommand{\qsfour}{{\sf QS4}}

\newcommand{\sfour}{{\sf S4}}
\newcommand{\sfive}{{\sf S5}}
\newcommand{\st}{{\sf S4.t}}
\newcommand{\qst}{{\sf QS4.t}}
\newcommand{\qost}{{\sf Q^\circ S4.t}}
\newcommand{\boxp}{\Box_P}
\newcommand{\boxf}{\Box_F}
\newcommand{\diap}{\Diamond_P}
\newcommand{\diaf}{\Diamond_F}
\newcommand{\cbff}{{\sf CBF_F}}
\newcommand{\cbfp}{{\sf CBF_P}}
\newcommand{\bff}{{\sf BF_F}}
\newcommand{\bfp}{{\sf BF_P}}
\newcommand{\cbarc}{{\sf CBF}}
\newcommand{\barc}{{\sf BF}}
\newcommand{\ui}{{\sf UI}}
\newcommand{\uio}{{\sf UI^\circ}}
\newcommand{\nid}{{\sf NID}}
\newcommand{\qk}{{\sf QK}}
\newcommand{\qok}{{\sf Q^\circ K}}

\newtheorem{theorem}{Theorem}[section]
\newtheorem*{theorem*}{Theorem}
\newtheorem{lemma}[theorem]{Lemma}
\newtheorem{proposition}[theorem]{Proposition}
\newtheorem*{proposition*}{Proposition}

\newtheorem*{corollary*}{Corollary}

\theoremstyle{definition}
\newtheorem{definition}[theorem]{Definition}
\newtheorem*{definition*}{Definition}

\newtheorem*{example*}{Example}

\newtheorem{remark}[theorem]{Remark}
\newtheorem*{remark*}{Remark}

\begin{document}

\title{Temporal interpretation of intuitionistic quantifiers}

\author[G. Bezhanishvili]{Guram Bezhanishvili}
\address{Department of Mathematical Sciences\\
New Mexico State University\\
Las Cruces NM 88003\\
USA}
\email{guram@nmsu.edu}

\author[L. Carai]{Luca Carai}
\address{Department of Mathematical Sciences\\
New Mexico State University\\
Las Cruces NM 88003\\
USA}
\email{lcarai@nmsu.edu}

\thanks{Acknowledgment: We would like to thank the reviewers whose comments have improved the presentation of the paper. }

\keywords{Intuitionistic quantifiers, temporal interpretation, G\"odel translation.}

  \begin{abstract}
We show that intuitionistic quantifiers admit the following temporal interpretation: $\forall x A$ is true at a world $w$ iff $A$ is true at every object in the domain of every future world, and $\exists x A$ is true at $w$ iff $A$ is true at some object in the domain of some past world. For this purpose we work with a predicate version of the well-known tense propositional logic $\sf S4.t$. The predicate logic $\sf Q^\circ S4.t$ is obtained by weakening the axioms of the standard predicate extension $\sf QS4.t$ of $\sf S4.t$ along the lines Corsi weakened $\sf QK$ to $\sf Q^\circ K$. The G\"odel translation embeds the predicate intuitionistic logic $\sf IQC$ into $\sf QS4$ fully and faithfully. We provide a temporal version of the G\"odel translation and prove that it embeds $\sf IQC$ into $\sf Q^\circ S4.t$ fully and faithfully; that is, we show that a sentence is provable in $\sf IQC$ iff its translation is provable in $\sf Q^\circ S4.t$. Faithfulness is proved using syntactic methods, while we prove fullness utilizing the generalized Kripke semantics of Corsi.
\end{abstract}

\maketitle

\section{Introduction}\label{sec:introduction}

Unlike classical connectives, intuitionistic connectives lack symmetry. It was noted already by McKinsey and Tarski \cite{MT46} that Heyting algebras (which are algebraic models of intuitionistic propositional calculus $\sf IPC$) are not symmetric even in the weak sense, meaning that the order-dual of a Heyting algebra may no longer be a Heyting algebra. In contrast, Boolean algebras (which are algebraic models of classical propositional calculus) are symmetric in the strong sense, meaning that the order-dual of a Boolean algebra $A$ is not only a Boolean algebra, but even isomorphic to $A$.

This non-symmetry has been addressed by several authors,
resulting in the concepts of bi-Heyting algebras and symmetric Heyting algebras.
Bi-Heyting algebras are obtained by adding to the signature of Heyting algebras a binary operation of co-implication, while symmetric Heyting algebras
by adding a de Morgan negation (and then co-implication becomes de Morgan dual of implication). The order-dual of a bi-Heyting algebra is again a bi-Heyting algebra, and the order-dual of a symmetric Heyting algebra $A$ is even isomorphic to $A$. Thus, the class of bi-Heyting algebras is symmetric in the weak sense, while the class of symmetric Heyting algebras is symmetric in the strong sense (hence the name).

The G\"odel translation of $\sf IPC$ into $\sf S4$ extends to a translation of the Heyting-Brouwer calculus $\sf HB$ of Rauszer \cite{Rau73} into the tense extension $\sf S4.t$ of $\sfour$, which has the future $\sf S4$-modality $\Box_F$ and the past $\sf S4$-modality $\Box_P$. The algebraic models of $\sf HB$ are bi-Heyting algebras, and implication is interpreted using $\Box_F$ and co-implication using $\Box_P$.

This story of non-symmetry also extends to intuitionistic quantifiers. Let $\iqc$ be the intuitionistic predicate calculus and $\qsfour$ the predicate $\sfour$. Not only the intuitionistic quantifiers $\forall x$ and $\exists x$ are not definable from each other (unlike the classical quantifiers), but the G\"odel translation $(\;)^t$ of $\iqc$ into $\qsfour$ is asymmetric in that  $(\forall x A)^t=\Box \forall x A^t$ and $(\exists x A)^t=\exists x A^t$. This is manifested in the interpretation of intuitionistic quantifiers in Kripke models. Indeed, a world $w$ of a Kripke model satisfies $\forall x A$ iff $A$ is true at every object of the domain $D_v$ of every world $v$ accessible from $w$, while $w$ satisfies $\exists x A$ iff $A$ is true at some object in the domain $D_w$ of $w$. If we think of the worlds of a Kripke model as ``states of knowledge," and the order between the states is temporal, then we can interpret the intuitionistic universal quantifier as ``for every object in the future," while the existential quantifier as ``for some object in the present."

In this article we present a more symmetric interpretation of intuitionistic quantifiers as ``for every object in the future" for $\forall x$ and ``for some object in the past" for $\exists x$. We show that such interpretation is supported by translating $\iqc$ fully and faithfully into a predicate tense logic by an appropriate modification of the G\"odel translation. As far as we know, this approach has not been considered in the past. One obvious obstacle is that it is unclear what predicate tense logic to choose for such a translation. Indeed, a natural candidate would be the standard predicate extension $\sf QS4.t$ of $\sf S4.t$. However, since $\sf QS4.t$ proves the Barcan formula, and hence the Kripke frames validating $\qst$ have constant domains, $\iqc$ does not translate fully into $\sf QS4.t$.
Instead we work with a weaker logic in which the universal instantiation axiom
\[
\forall x A \to A(y/x)
\]
is replaced by a weaker version
\[
\forall y (\forall x A \to A(y / x)).
\]
This
approach is along the lines of Kripke~\cite{Kri63}, Hughes and Cresswell~\cite{CH96}, Fitting and Mendelsohn~\cite{FM98}, and Corsi~\cite{Cor02} who considered modal predicate logics without the Barcan and/or converse Barcan formulas.
The generalized Kripke frames considered in this semantics have two domains associated to each world, an inner domain and an outer domain. The inner domains are always contained in the outer domains and are not necessarily increasing. While variables are interpreted in the outer domains, the scope of quantifiers is restricted to the inner domains.
Utilizing this approach, we define a tense predicate logic $\qost$ which is sound with respect to the generalized Kripke semantics with nonempty increasing inner domains and constant outer domains.
We modify the G\"odel translation to define a temporal translation of $\iqc$ into $\qost$
as follows:
\begin{equation*}
\begin{array}{r c l l}
\bot^t &=& \bot & \\
P(x_1, \ldots, x_n)^t &=& \boxf P(x_1, \ldots, x_n) & \qquad \text{for each n-ary predicate symbol } P \\
(A \land B)^t &=& A^t \land B^t &\\
(A \lor B)^t &=& A^t \lor B^t &\\
(A \to B)^t &=& \boxf (A^t \to B^t) &\\
(\forall x A)^t &=& \boxf \forall x A^t & \\
(\exists x A)^t &=& \diap \exists x A^t &
\end{array}
\end{equation*}
Here $\boxf$ is the $\sfour$-modality interpreted as ``always in the future" and $\diap$ is the $\sfour$-modality interpreted as ``sometime in the past."
Thus, the modification of the G\"odel translation concerns the clause for $\exists x A$.
Our main result states that this translation is full and faithful in the following sense:\\

{\bf Main Theorem.}
\begin{itemize}
\item For any formula $A$ in the language of $\iqc$, we have
\begin{equation*}
\iqc \vdash A \quad \mbox{iff} \quad \qost \vdash \forall x_1 \cdots \forall x_n A^t
\end{equation*}
where $x_1, \ldots, x_n$ are the free variables in $A$.
\item If $A$ is a sentence, then
\begin{equation*}
\iqc \vdash A \quad \mbox{iff} \quad \qost \vdash A^t.
\end{equation*}
\end{itemize}

The proof of this surprising result is along the lines of the standard proof of fullness and faithfulness of the G\"odel translation of $\iqc$ into $\qsfour$. We would like to stress that the main challenge is not so much the proof itself, but rather finding the ``right" predicate tense modal logic into which to translate $\iqc$. We find it of interest to explore philosophical (as well as practical) consequences of this new temporal point of view on $\iqc$.

The paper is structured as follows. In Section~\ref{sec:iqc} we recall the intuitionistic predicate logic $\iqc$ and its Kripke completeness. In Section~\ref{sec:modal predicate} we briefly review the basics of modal predicate logics and their Kripke semantics, including weaker modal predicate logics.
In Section~\ref{sec:qost} we recall the tense propositional logic $\st$, consider its standard predicate extension $\qst$, and then introduce its weakening $\qost$ which is our main tense predicate logic of interest.
We conclude the section by observing that $\qost$ is sound with respect to a version of the generalized Kripke semantics studied by Kripke \cite{Kri63}, Hughes and Cresswell \cite{CH96}, Fitting and Mendelsohn \cite{FM98}, and Corsi \cite{Cor02}.
Our main result, that $\iqc$ embeds into $\qost$ fully and faithfully, is proved in Section~\ref{sec:translation}.
We prove faithfulness syntactically, while fullness is proved semantically.
We conclude the paper with Section~\ref{sec:open problems} in which we describe some open problems our study has generated.
Finally, the Appendix
contains the proofs of some technical lemmas used in Sections~\ref{sec:qost} and~\ref{sec:translation}.

\section{The intuitionistic predicate logic}\label{sec:iqc}

Let $\iqc$ be the intuitionistic predicate logic. We recall that the language $\mathcal{L}$ of $\iqc$ consists of countably many individual variables $x,y, \ldots$, countably many $n$-ary predicate symbols $P,Q, \ldots$ (for each $n \geq 0$), the logical connectives $\bot,\wedge,\vee,\to$, and the quantifiers $\forall,\exists$.
We do not add any constants to $\mathcal{L}$ since this results in the temporal translation not being faithful (see Remark~\ref{rem:adding constants}).

Formulas are defined as usual by induction and are denoted with upper case letters $A,B, \ldots$.
Let $x,y$ be individual variables and $A$ a formula. If $x$ is a free variable of $A$ and does not occur in the scope of $\forall y$ or $\exists y$, then we denote by $A(y/x)$ the formula obtained from $A$ by replacing all the free occurrences of $x$ by $y$.

The following definition of $\iqc$ is taken from \cite[Sec~2.6]{GSS09}. We point out that, unlike \cite{GSS09}, we prefer to work with axiom schemes, and hence do not need the inference rule of substitution.

\begin{definition}\label{def:iqc}
The intuitionistic predicate logic $\iqc$ is the least set of formulas of $\mathcal{L}$ containing all substitution instances of theorems of $\ipc$, the axiom schemes
\begin{enumerate}
\item $\forall x A \to A(y/x)$ \hfill Universal instantiation $(\ui)$
\item $A(y/x) \to \exists x A$
\item $\forall x (A \to B) \to (A \to \forall x B)$ \ with $x$ not free in $A$
\item $\forall x (A \to B) \to (\exists x A \to B)$ \ with $x$ not free in $B$
\end{enumerate}
and closed under the inference rules
\vspace{0.2ex}
\[
\begin{array}{c p{0.14cm} l p{0.7cm} c p{0.14cm} l}
\inference {A \quad A \to B}{B} & & \mbox{Modus Ponens (MP)} & & \inference {A}{\forall x A} & & \mbox{Generalization (Gen)} \\[2ex]
\end{array}
\]
\end{definition}

We next describe Kripke semantics for $\iqc$ (see \cite{Kri65,Gab81}).

\begin{definition}\label{def:iqc frame}
An \emph{$\iqc$-frame} is a triple $\mathfrak{F}=(W,R,D)$ where
\begin{itemize}
\item $W$ is a nonempty set whose elements are called the \emph{worlds} of $\mathfrak{F}$.
\item $R$ is a partial order on $W$.
\item $D$ is a function that associates to each $w \in W$ a nonempty set $D_w$ such that $w R v$ implies $D_w \subseteq D_v$ for each $w,v \in W$. The set $D_w$ is called the \emph{domain} of $w$.
\end{itemize}
\end{definition}

\begin{definition}
\begin{itemize}
\item[]
\item An \emph{interpretation} of $\mathcal{L}$ in $\mathfrak{F}$ is a function $I$ associating to each world $w$ and any $n$-ary predicate symbol $P$ an $n$-ary relation $I_w(P) \subseteq (D_w)^n$ such that $w R v$ implies $I_w(P) \subseteq I_v(P)$.
\item A \emph{model} is a pair $\mathfrak{M}=(\mathfrak{F},I)$ where  $\mathfrak{F}$ is an $\iqc$-frame and $I$ is an interpretation in $\mathfrak{F}$.
\item Let $w$ be a world of $\mathfrak{F}$. A \emph{$w$-assignment} is a function $\sigma$ associating to each individual variable $x$ an element $\sigma(x)$ of $D_w$. Note that if $wRv$, then $\sigma$ is also a $v$-assignment.
\item Let $\sigma$ and $\tau$ be two $w$-assignments and $x$ an individual variable. Then $\tau$ is said to be an \emph{$x$-variant} of $\sigma$ if $\tau(y)=\sigma(y)$ for all $y \neq x$.
\end{itemize}
\end{definition}

We next recall the definition of when a formula $A$ is true in a world $w$ of a model $\mathfrak{M}=(\mathfrak{F}, I)$ under the $w$-assignment $\sigma$, written
$\mathfrak{M} \vDash_w^\sigma A$.

\begin{definition}
\begin{equation*}
\begin{array}{l c l}
\mathfrak{M} \vDash_w^\sigma \bot &  & \mbox{never} \\
\mathfrak{M} \vDash_w^\sigma P(x_1, \ldots, x_n) & \text{ iff } & (\sigma(x_1), \ldots, \sigma(x_n)) \in I_w(P) \\
\mathfrak{M} \vDash_w^\sigma B \land C & \text{ iff } & \mathfrak{M} \vDash_w^\sigma B \mbox{ and } \mathfrak{M} \vDash_w^\sigma C \\
\mathfrak{M} \vDash_w^\sigma B \lor C & \text{ iff } & \mathfrak{M} \vDash_w^\sigma B \mbox{ or } \mathfrak{M} \vDash_w^\sigma C \\
\mathfrak{M} \vDash_w^\sigma B \to C & \text{ iff } & \mbox{for all $v$ with $wRv$, if $\mathfrak{M} \vDash_v^\sigma B$, then $\mathfrak{M} \vDash_v^\sigma C$}\\
\mathfrak{M} \vDash_w^\sigma \forall x B & \text{ iff } & \mbox{for all $v$ with $wRv$ and each $v$-assignment $\tau$} \\
& & \mbox{that is an $x$-variant of $\sigma$, } \mathfrak{M} \vDash_v^\tau B \\
\mathfrak{M} \vDash_w^\sigma \exists x B & \text{ iff } & \mbox{there exists a $w$-assignment $\tau$} \\
& & \mbox{that is an $x$-variant of $\sigma$ such that } \mathfrak{M} \vDash_w^\tau B
\end{array}
\end{equation*}
\end{definition}

\begin{definition}\label{def:truth and validity in iqc}
\begin{itemize}
\item[]
\item We say that $A$ is \emph{true} in a world $w$ of $\mathfrak{M}$, written $\mathfrak{M} \vDash_w A$, if for all $w$-assignments $\sigma$, we have $\mathfrak{M} \vDash_w^\sigma A$.
\item We say that $A$ is \emph{true} in $\mathfrak{M}$, written $\mathfrak{M} \vDash A$, if for all worlds $w \in W$, we have $\mathfrak{M} \vDash_w A$.
\item We say that $A$ is \emph{valid} in a frame $\mathfrak{F}$, written $\mathfrak{F} \vDash A$, if for all models $\mathfrak{M}$ based on $\mathfrak{F}$, we have $\mathfrak{M} \vDash A$.
\end{itemize}
\end{definition}

We have the following well-known completeness of $\iqc$ with respect to Kripke semantics.

\begin{theorem} [\cite{Kri65}]\label{thm:sound and compl of iqc}
The intuitionistic predicate logic $\iqc$ is sound and complete with respect to Kripke semantics; that is, for each formula $A$,
\[
\iqc \vdash A \mbox{ iff }\mathfrak{F} \vDash A \mbox{ for each $\iqc$-frame $\mathfrak{F}$.}
\]
\end{theorem}

\section{Modal predicate logics}\label{sec:modal predicate}

Modal predicate logics were first studied by Barcan~\cite{Bar46} and Carnap~\cite{Car46} in 1940s. The semantic study of modal predicate logics was initiated by Kripke~\cite{Kri59,Kri63} in late 1950s/early 1960s. Since then many completeness results have been obtained with respect to Kripke semantics, but there is also a large body of incompleteness results, which is one of the reasons that the model theory of modal predicate logics is less advanced than that of modal propositional logics (see, e.g.,~\cite{GSS09,Gar01} and the references therein).

Let $\sf K$ be the least normal modal propositional logic and let $\qk$ be the standard predicate extension of $\sf K$. The language $\mathcal{L}_\Box$ of $\qk$ is the extension of $\mathcal L$ with the modality $\Box$.
Since the modal logics we consider are based on the classical logic, it is sufficient to only consider the logical connectives $\bot,\to$ and the quantifier $\forall$. The logical connectives $\wedge,\vee,\neg,\leftrightarrow$, the quantifier $\exists$, and the modality $\Diamond$ are treated as usual abbreviations.

We next recall the definition of $\qk$ (see, e.g.,~\cite[Sec~2.6]{GSS09}, but note, as in Section~\ref{sec:iqc}, that we work with axiom schemes instead of having the inference rule of substitution).

\begin{definition}\label{def:qk}
The modal predicate logic $\qk$ is the least set of formulas of $\mathcal{L}_\Box$ containing all substitution instances of theorems of $\sf K$, the axiom schemes (i) and (iii) of Definition~\ref{def:iqc},
and closed under (MP), (Gen), and
\begin{equation*}
\inference {A}{\Box A} \quad \mbox{Necessitation (N)}
\end{equation*}
\end{definition}

The definition of \emph{$\qk$-frames} $\mathfrak{F}=(W,R,D)$ is the same as that of $\iqc$-frames (see Definition~\ref{def:iqc frame}) with the only difference that $R$ can be an arbitrary relation. \emph{Models} are also defined the same way, but without the requirement that $w R v$ implies $I_w(P) \subseteq I_v(P)$. The connectives and quantifiers are interpreted at each world in the usual classical way, and
\[
\mathfrak M\models_w^\sigma\Box A \mbox{ iff } (\forall v\in W)(w R v \Rightarrow \mathfrak M\models_v^\sigma A).
\]
\emph{Truth} and \emph{validity} of formulas are defined as usual.

We next give a brief history of first Kripke completeness results for modal predicate logics. In 1959 Kripke~\cite{Kri59} proved Kripke completeness of predicate $\sfive$. In late 1960s Cresswell~\cite{Cre67,Cre68} (see also Hughes and Cresswell~\cite{CH68}), Sch\"utte~\cite{Sch68}, and Thomason~\cite{Tho70} proved Kripke completeness of predicate $\sf T$ and $\sfour$. Kripke completeness of $\qk$ was first established by Gabbay~\cite[Thm.~8.5]{Gab76}\footnote{We would like to thank Ilya Shapirovsky and Valentin Shehtman for useful discussions on the history of Kripke completeness for modal predicate logics.}:

\begin{theorem} \label{thm:qk}
The modal predicate logic $\qk$ is sound and complete with respect to Kripke semantics.
\end{theorem}

The following two principles play an important role in the study of modal predicate logics. They were first considered by Barcan~\cite{Bar46}.
\begin{equation*}
\begin{array}{l p{1cm} l p{0.5cm} l}
\Box \forall x A \to \forall x \Box A & & \mbox{converse Barcan formula} & & (\cbarc)\\
\forall x \Box A \to \Box \forall x A & & \mbox{Barcan formula} & & (\barc)\\
\end{array}
\end{equation*}
It is easy to see that $\cbarc$ is a theorem of $\qk$.
Indeed, this follows from Theorem~\ref{thm:qk} and the fact that domains of each $\qk$-frame are increasing. On the other hand, a $\qk$-frame validates $\barc$ iff it has \emph{constant domains}, meaning that $wRv$ implies $D_w=D_v$, and we have the following well-known theorem (see, e.g.,~\cite[Thm.~9.3]{Gab76}):

\begin{theorem}
The logic $\qk+\barc$ is sound and complete with respect to the class of $\qk$-frames with constant domains.
\end{theorem}

A modal predicate logic whose Kripke frames have neither increasing nor decreasing domains was considered already by Kripke~\cite{Kri63}.
Building on this work,
Hughes and Cresswell~\cite[pp.~304--309]{CH96} introduced a similar predicate modal logic and proved its completeness with respect to a generalized Kripke semantics. Fitting and Mendelsohn~\cite[Sec.~6.2]{FM98} gave an alternate axiomatization of this logic. Building on the work of Fitting and Mendelsohn,
Corsi~\cite{Cor02} defined the system $\qok$ whose axiomatization contains a weakening of the universal instantiation axiom.

\begin{definition}\label{def:qok}
The logic $\qok$ is the least set of formulas of $\mathcal{L}_\Box$ containing all substitution instances of theorems of
$\sf K$, the axiom schemes
\begin{enumerate}
\item $\forall y (\forall x A \to A(y / x))$ \hfill $(\uio)$
\item $\forall x (A \to B) \to (\forall x A \to \forall x B)$
\item $\forall x \forall y A \leftrightarrow \forall y \forall x A$
\item $A \rightarrow \forall x A$ \ with $x$ not free in $A$
\end{enumerate}
and closed under (MP), (Gen), and (N).
\end{definition}

\begin{remark}
In Definition~\ref{def:qok}, replacing $\uio$ with $\ui$ yields an equivalent definition of $\qk$. Therefore, $\qok$ is contained in $\qk$.
\end{remark}

Kripke frames for $\qok$ generalize Kripke frames for $\qk$ by having two domains, inner and outer.

\begin{definition} \label{def:qok frame}
A \textit{$\qok$-frame} is a quadruple $\mathfrak{F}=(W,R,D,U)$ where
\begin{itemize}
\item $(W,R)$ is a $\sf K$-frame.
\item $D$ is a function that associates to each $w \in W$ a set $D_w$.
The set $D_w$ is called the \emph{inner domain} of $w$.
\item $U$ is a nonempty set containing the union of all the $D_w$. The set $U$ is called the \emph{outer domain} of $\mathfrak{F}$.
\end{itemize}
\end{definition}

Definition~\ref{def:qok frame} is a particular case of the frames considered by Corsi \cite{Cor02}
where increasing outer domains are allowed.
For our purposes, taking a fixed outer domain $U$ is sufficient.
We recall from \cite{Cor02} how to interpret $\mathcal{L}_\Box$ in a $\qok$-frame $\mathfrak{F}=(W,R,D,U)$.

\begin{definition}\label{def:model and assign for qok}
\begin{itemize}
\item[]
\item An \emph{interpretation} of $\mathcal{L}_\Box$ in $\mathfrak{F}$ is a function $I$ associating to each world $w$ and an $n$-ary predicate symbol $P$ an $n$-ary relation $I_w(P) \subseteq U^n$.
\item A \emph{model} is a pair $\mathfrak{M}=(\mathfrak{F},I)$ where $\mathfrak{F}$ is a $\qok$-frame and $I$ is an interpretation in $\mathfrak{F}$.
\item An \emph{assignment} in $\mathfrak{F}$ is a function $\sigma$ that associates to each individual variable an element of $U$.
\item If $\sigma$ and $\tau$ are two assignments and $x$ is an individual variable, $\tau$ is said to be an \emph{$x$-variant} of $\sigma$ if $\tau(y)=\sigma(y)$ for all $y \neq x$.
\item We say that an assignment $\sigma$ is \emph{$w$-inner} for $w \in W$ if $\sigma(x) \in D_w$ for each individual variable $x$.
\end{itemize}
\end{definition}

We next recall from \cite{Cor02} the definition of when a formula $A$ is true in a world $w$ of a model $\mathfrak{M}=(\mathfrak{F}, I)$ under the assignment $\sigma$, written
$\mathfrak{M} \vDash_w^\sigma A$.

\begin{definition}\label{def:truth in a model for qok}
\begin{equation*}
\begin{array}{l c l}
\mathfrak{M} \vDash_w^\sigma \bot &  & \mbox{never} \\
\mathfrak{M} \vDash_w^\sigma P(x_1, \ldots, x_n) & \text{ iff } & (\sigma(x_1), \ldots, \sigma(x_n)) \in I_w(P) \\
\mathfrak{M} \vDash_w^\sigma B \to C & \text{ iff } & \mathfrak{M} \vDash_w^\sigma B \text{ implies } \mathfrak{M} \vDash_w^\sigma C \\
\mathfrak{M} \vDash_w^\sigma \forall x B & \text{ iff } & \mbox{for all $x$-variants $\tau$ of $\sigma$ with $\tau(x) \in D_w$, } \mathfrak{M} \vDash_w^\tau B \\
\mathfrak{M} \vDash_w^\sigma \Box B & \text{ iff } & \mbox{for all $v$ such that $wRv$, } \mathfrak{M} \vDash_v^\sigma B \\
\end{array}
\end{equation*}
\end{definition}

\begin{definition}\label{def:truth and validity in qok}
A formula $A$ is \emph{true} in a model $\mathfrak{M}=(\mathfrak{F}, I)$ at the world $w \in W$ (in symbols $\mathfrak{M} \vDash_w A$) if for all assignments $\sigma$, we have $\mathfrak{M} \vDash_w^\sigma A$.
The definition of \emph{truth} in a model and \emph{validity} in a frame are the same as in Definition~\ref{def:truth and validity in iqc}.
\end{definition}

We have the following completeness result for $\qok$, see~\cite[Thm.~1.32]{Cor02} and its proof.

\begin{theorem}
$\qok$ is sound and complete with respect to the class of $\qok$-frames.
\end{theorem}

\begin{definition}
Let $\mathfrak{F}=(W,R,D,U)$ be a $\qok$-frame.
\begin{itemize}
\item We say that $\mathfrak{F}$ has \emph{increasing inner domains} if $wRv$ implies $D_w \subseteq D_v$ for each $w,v\in W$.
\item We say that $\mathfrak{F}$ has \emph{decreasing inner domains} if $wRv$ implies $D_v \subseteq D_w$ for each $w,v\in W$.
\item If $\mathfrak{F}$ has both increasing and decreasing inner domains, we say that it has \emph{constant inner domains}.
\end{itemize}
\end{definition}

The following axiom scheme guarantees nonempty inner domains (hence the abbreviation):
\begin{equation*}
\begin{array}{l p{2cm} l}
\forall x A \rightarrow A \: \mbox{ with $x$ not free in $A$} &  & (\nid)\\
\end{array}
\end{equation*}

The next proposition is not difficult to verify (see, e.g.,~\cite[Sec.~4.9]{FM98} and~\cite[pp.~1487--1488]{Cor02}).

\begin{proposition} 
Let $\mathfrak{F}=(W,R,D,U)$ be a $\qok$-frame.
\begin{itemize}
\item $\mathfrak{F}$ validates $\cbarc$ iff $\mathfrak{F}$ has increasing inner domains.
\item $\mathfrak{F}$ validates $\barc$ iff $\mathfrak{F}$ has decreasing inner domains.
\item $\mathfrak{F}$ validates $\nid$ iff $\mathfrak{F}$ has nonempty inner domains.
\end{itemize}
\end{proposition}

We have the following completeness results for logics obtained by adding $\cbarc$, $\barc$, and $\nid$ to $\qok$ (see~\cite[Thms.~1.30, 1.32, and Footnote~7]{Cor02}):

\begin{theorem}\label{thm:completeness results from corsi}
\begin{itemize}
\item[]
\item $\qok+\cbarc$ is sound and complete with respect to the class of $\qok$-frames with increasing inner domains.
\item $\qok+\cbarc+\barc$ is sound and complete with respect to the class of $\qok$-frames with constant inner domains.
\item Adding $\nid$ to the above two logics or to $\qok$ yields completeness of the resulting logics with respect to the corresponding classes of frames which have nonempty inner domains.
\end{itemize}
\end{theorem}

On the other hand, completeness of $\qok + \barc$ remains open (see~\cite[p.~1510]{Cor02}).

\section{The logic $\qost$}\label{sec:qost}

The tense predicate logic we will translate $\iqc$ into is based on the well-known tense propositional logic $\st$.
We use $\boxf$ (``always in the future'') and $\boxp$ (``always in the past'') as temporal modalities. Then
$\diaf$ (``sometime in the future'') and $\diap$ (``sometime in the past'') are usual abbreviations $\neg \boxf \neg$ and $\neg \boxp \neg$.

\begin{definition}
The logic $\st$ is the least set of formulas of the tense propositional language containing all substitution instances of $\sfour$-axioms for both $\boxf$ and $\boxp$, the axiom schemes
\begin{enumerate}
\item $A \to \boxp \diaf A$
\item $A \to \boxf \diap A$
\end{enumerate}
and closed under (MP) and
\[
\begin{array}{c p{0.2cm} l p{0.8cm} c p{0.2cm} l}
\inference {A}{\boxf A} & & \mbox{$\boxf$-Necessitation (N\textsubscript{F})} & & \inference {A}{\boxp A}  & & \mbox{$\boxp$-Necessitation (N\textsubscript{P})}
\end{array}
\]
\end{definition}

Relational semantics of $\st$ consists of Kripke frames $\mathfrak{F}=(W,R)$ where $R$ is reflexive and transitive. As usual, propositional letters are interpreted as subsets of $W$, classical connectives as the corresponding set-theoretic operations on the powerset of $W$,
and for temporal modalities we set:
\begin{equation*}
\begin{array}{l c l}
w \vDash \boxf A & \text{ iff } & (\forall v \in W)(wRv \Rightarrow v \vDash A) \\
w \vDash \boxp A & \text{ iff } & (\forall v\in W)(vRw \Rightarrow v \vDash A)
\end{array}
\end{equation*}

It is well known that $\st$ is sound and complete with respect to its relational semantics. 

Let $\mathcal{L}_T$ be the bimodal predicate language obtained by extending $\mathcal{L}$ with two modalities $\boxf$ and $\boxp$.

\begin{definition}
The logic $\qst$ is the least set of formulas of $\mathcal{L}_T$ containing all substitution instances of theorems of $\st$, the axiom schemes (i) and (iii) of Definition~\ref{def:iqc},
and closed under  (MP), (Gen), (N\textsubscript{F}), and (N\textsubscript{P}).
\end{definition}

The following are temporal versions of $\cbarc$ and $\barc$:
\begin{equation*}
\begin{array}{l p{1cm} l p{0.5cm} l}
\boxf \forall x A \to \forall x \boxf A & & \mbox{converse Barcan formula for $\boxf$} & & (\cbff)\\
\forall x \boxf A \to \boxf \forall x A & & \mbox{Barcan formula for $\boxf$} & & (\bff)\\
\boxp \forall x A \to \forall x \boxp A & & \mbox{converse Barcan formula for $\boxp$} & & (\cbfp)\\
\forall x \boxp A \to \boxp \forall x A & & \mbox{Barcan formula for $\boxp$} & & (\bfp)
\end{array}
\end{equation*}
The proof that $\qk\vdash\cbarc$ (see, e.g.,~\cite[p.~88]{Kri63}) can be adapted to prove that $\qst\vdash\cbff$ and $\qst\vdash\cbfp$. It is also well known that
$\cbff$ and $\bfp$, as well as $\cbfp$ and $\bff$ are derivable from each other in any tense predicate logic.
Therefore, all four are theorems of $\qst$. 
This is reflected in the fact that $\qst$-frames have constant domains.
Indeed, $\qst$ is complete with respect to this semantics (see Section~\ref{sec:open problems}).
But this is problematic for translating $\iqc$ fully into $\qst$ since $\iqc$-frames with constant domains validate the additional axiom $\forall x (A \lor B) \to (A \lor \forall x B)$, where $x$ is not free in $A$, which is not a theorem of $\iqc$ (see, e.g.,~\cite[p.~53, Cor.~8]{Gab81}).

Consequently, we need to work with a weaker logic than $\qst$. To this end, we introduce the logic $\qost$, which weakens $\qst$ the same way $\qok$ weakens $\qk$.

\begin{definition}\label{def:qost}
The logic $\qost$ is the least set of formulas of $\mathcal{L}_T$ containing all substitution instances of theorems of $\st$, the axiom schemes
(i), (ii), (iii), (iv) of $\qok$ (see Definition~\ref{def:qok}), $\nid$, $\cbff$,
and closed under (MP), (Gen), (N\textsubscript{F}), and (N\textsubscript{P}).
\end{definition}

As follows from Proposition~\ref{prop:bfp in qost} in the Appendix, $\bfp$ is a theorem of $\qost$. In fact,
$\cbff$ and
$\bfp$ are derivable from each other and the other axioms of $\qost$.

\begin{definition}
A \emph{$\qost$-frame} is a $\qok$-frame $\mathfrak F=(W,R,D,U)$ (see Definition~\ref{def:qok frame}) with nonempty increasing inner domains whose accessibility relation is reflexive and transitive.
\end{definition}

Models and assignments
are defined as in Definition~\ref{def:model and assign for qok}.
The clauses of
when a formula $A$ of $\mathcal{L}_T$ is true in a world $w$ of a $\qost$-model $\mathfrak{M}=(\mathfrak{F}, I)$ under the assignment $\sigma$, written
$\mathfrak{M} \vDash_w^\sigma A$, are defined as in Definition~\ref{def:truth in a model for qok}, but we replace the $\Box$-clause with the following two clauses:

\begin{equation*}
\begin{array}{l c l}
\mathfrak{M} \vDash_w^\sigma \boxf B & \text{ iff } & (\forall v\in W)(wRv \Rightarrow \mathfrak{M} \vDash_v^\sigma B) \\
\mathfrak{M} \vDash_w^\sigma \boxp B & \text{ iff } & (\forall v\in W)(vRw \Rightarrow \mathfrak{M} \vDash_v^\sigma B)
\end{array}
\end{equation*}

For formulas of $\mathcal{L}_T$ we define truth in a model and validity in a frame as in Definition~\ref{def:truth and validity in qok}.

\begin{theorem}\label{prop:soundness qost}
$\qost$ is sound with respect to the class of $\qost$-frames; that is, for each formula $A$ of $\mathcal{L}_T$ and $\qost$-frame $\mathfrak{F}$, from $\qost \vdash A$ it follows that $\mathfrak{F} \vDash A$.
\end{theorem}

\begin{proof}
It is sufficient to show that each axiom scheme is valid in all $\qost$-frames and that each rule of inference preserves validity. This can be done by direct verification. We only
show that the axiom scheme $\cbff$ is valid in all $\qost$-frames. Let $\mathfrak{M}=(\mathfrak{F},I)$ be a $\qost$-model, $w \in W$, and $\sigma$ an assignment. If $\mathfrak{M} \vDash_w^\sigma \boxf \forall x A$, then for all $v$ with $wRv$ we have $\mathfrak{M} \vDash_v^\sigma \forall x A$. This implies that for each $x$-variant $\tau$ of $\sigma$ with $\tau(x) \in D_v$ we have $\mathfrak{M} \vDash_v^\tau A$. Since $D_w \subseteq D_v$, this is in particular true for $x$-variants $\tau$ of $\sigma$ with $\tau(x) \in D_w$. Therefore, for each $x$-variant $\tau$ of $\sigma$ with $\tau(x) \in D_w$ and for each $v$ with $wRv$ we have $\mathfrak{M} \vDash_v^\tau A$. Thus, for each $x$-variant $\tau$ of $\sigma$ with $\tau(x) \in D_w$, we have $\mathfrak{M} \vDash_w^\sigma \boxf A$. Consequently, $\mathfrak{M} \vDash_w^\sigma \forall x \boxf A$.
This shows that $\mathfrak{F} \vDash \boxf \forall x A \to \forall x \boxf A$ for each $\qost$-frame $\mathfrak{F}$.
\end{proof}

On the other hand, completeness of $\qost$
remains an interesting open problem, which is related to the open problem of completeness of $\qok+\barc$ (see Section~\ref{sec:open problems}).

\section{The translation}\label{sec:translation}

In this section we prove our main result that the temporal modification (described in the Introduction) of the G\"odel translation  
embeds $\iqc$ into $\qost$ fully and faithfully.
Our strategy is to prove faithfulness of the translation syntactically, while fullness will be proved by semantical means, utilizing Kripke completeness of $\iqc$.

Our syntactic proof of faithfulness is based on the following technical lemma, the proof of which we give in the Appendix.
To keep the notation simple,
we denote lists of variables by bold letters. If $\mathbf{x}=x_1, \ldots, x_n$, we write $\forall \mathbf{x}$ for $\forall x_1 \cdots \forall x_n$.
We point out that it is a consequence of axioms~(ii) and~(iii) of $\qok$ 
that from the point of view of provability in $\qost$, the order of variables in $\forall\mathbf{x}$ does not matter.

\begin{lemma}\label{lem:faith facts without proof}
\begin{enumerate}
\item[]
\item Let $C$ be an instance of an axiom scheme of $\iqc$ and $\mathbf{x}$ the list of free variables in $C$. Then $\qost \vdash \forall \mathbf{x} \, C^t$.
\item Let $A,B$ be formulas of $\mathcal{L}$, $\mathbf{x}$ the list of variables free in ${A \to B}$, $\mathbf{y}$ the list of variables free in $A$, and $\mathbf{z}$ the list of variables free in $B$. If $\qost \vdash \forall \mathbf{x} (A \to B)^t$ and $\qost \vdash \forall \mathbf{y} A^t$, then $\qost \vdash \forall \mathbf{z} B^t$.
\item Let $A$ be a formula of $\mathcal{L}$, $x$ a variable, $\mathbf{y}$ the list of variables free in $A$, and $\mathbf{z}$ the list of variables free in $\forall x A$. If $\qost \vdash \forall \mathbf{y} A^t$, then $\qost \vdash \forall \mathbf{z} \, (\forall x A)^t$.
\end{enumerate}
\end{lemma}

\begin{proof}
For (i) see the proof of Lemma~\ref{lem:faithful axiom}, for (ii) see the proof of Lemma~\ref{lem:faithful MP}, and for (iii) see the proof of Lemma~\ref{lem:faithful gen}.
\end{proof}

\begin{theorem}\label{thm: faith formula}
Let $A$ be a formula of $\mathcal{L}$ and $x_1, \ldots, x_n$ the free variables of $A$. If $\iqc \vdash A$, then $\qost \vdash \forall x_1 \cdots \forall x_n A^t$.
\end{theorem}

\begin{proof}
The proof is by induction on the length of the proof of $A$ in $\iqc$.
If $A$ is an instance of an axiom of $\iqc$, then the result follows from Lemma~\ref{lem:faith facts without proof}(i).
Lemma~\ref{lem:faith facts without proof}(ii) takes care of the case in which the last step of the proof of $A$ is an application of (MP). Finally, if the last step of the proof of $A$ is an application of (Gen) to the variable $x$, use Lemma~\ref{lem:faith facts without proof}(iii).
\end{proof}

\begin{remark}\label{rem:need univ closure}
We are prefixing the translation of $A$ with $\forall x_1 \cdots \forall x_n$ because it is not true in general that $\iqc \vdash A$ implies $\qost \vdash A^t$. For example, if $A$ is an instance of the universal instantiation axiom, which is an axiom of $\iqc$, then $A^t$ is not in general
a theorem of $\qost$.
\end{remark}

\begin{definition}
\begin{itemize}
\item[]
\item For an $\iqc$-frame $\mathfrak{F}=(W,R,D)$ let $\overline{\mathfrak{F}}=(W,R,D,U)$ where $U=\bigcup \{ D_w \mid w \in W \}$.
\item For an $\iqc$-model $\mathfrak{M}=(\mathfrak{F},I)$ let $\overline{\mathfrak{M}}=(\overline{\mathfrak{F}},I)$.
\end{itemize}
\end{definition}

\begin{remark}\label{rem:mbar}
\begin{itemize}
\item[]
\item It is obvious that $\overline{\mathfrak{F}}$ is a $\qost$-frame.
\item If $I$ is an interpretation in $\mathfrak{F}$, then $I$ is also an interpretation in $\overline{\mathfrak{F}}$ because for each $n$-ary predicate letter $P$ we have $I_w(P) \subseteq D_w^n \subseteq U^n$. Therefore, $\overline{\mathfrak{M}}$ is well defined.
\item The $w$-assignments in $\mathfrak{F}$ are exactly the $w$-inner assignments in $\overline{\mathfrak{F}}$.
\end{itemize}
\end{remark}

The proof of the following technical lemma is given in the Appendix.

\begin{lemma}\label{lem:5.6}
If $A$ is a formula of $\mathcal{L}$, then $\qost \vdash A^t \to \boxf A^t$.
\end{lemma}

\begin{proof}
See the proof of Lemma~\ref{lem:diap At to At}.
\end{proof}

\begin{lemma} \label{lem:A^t truth increasing}
Let $A$ be a formula of $\mathcal{L}$, $\mathfrak{M}=(\mathfrak{F}, I)$ a $\qost$-model, and $\sigma$ an assignment in $\mathfrak{F}$.
If $v,w \in W$ with $v R w$, then $\mathfrak{M} \vDash_v^\sigma A^t$ implies $\mathfrak{M} \vDash_w^\sigma A^t$.
\end{lemma}

\begin{proof}
Suppose $vRw$ and $\mathfrak{M} \vDash_v^\sigma A^t$. By Lemma~\ref{lem:5.6} and Theorem~\ref{prop:soundness qost}, $\mathfrak{M} \vDash_v^\sigma A^t \to \boxf A^t$. Therefore,
$\mathfrak{M} \vDash_v^\sigma \boxf A^t$, which yields $\mathfrak{M} \vDash_w^\sigma A^t$ because $vRw$.
\end{proof}

\begin{proposition} \label{prop:validity and mbar}
Let $A$ be a formula of $\mathcal{L}$, $\mathfrak{M}=(\mathfrak{F}, I)$ an $\iqc$-model based on an $\iqc$-frame $\mathfrak{F}=(W,R,D)$, and $w \in W$.
\begin{enumerate}
\item For each $w$-assignment $\sigma$,
\begin{equation*}
\mathfrak{M} \vDash_{w}^\sigma A \mbox{ iff } \overline{\mathfrak{M}} \vDash_{w}^\sigma A^t.
\end{equation*}
\item If $x_1, \ldots, x_n$ are the free variables of $A$, then
\begin{equation*}
\mathfrak{M} \vDash_{w} A \mbox{ iff } \overline{\mathfrak{M}} \vDash_{w} \forall x_1 \cdots \forall x_n A^t.
\end{equation*}
\end{enumerate}
\end{proposition}

\begin{proof}
(i). Induction on the complexity of $A$. Let $A$ be an atomic formula $P(x_1, \ldots, x_n)$. Since $w R v$ implies $I_w(P) \subseteq I_v(P)$ and $R$ is reflexive, we have
\begin{align*}
\mathfrak{M} \vDash_w^\sigma P(x_1, \ldots, x_n) & \mbox{ iff } (\sigma(x_1), \ldots, \sigma(x_n)) \in I_w(P) \\
& \mbox{ iff } (\forall v\in W)(w R v \Rightarrow (\sigma(x_1), \ldots, \sigma(x_n)) \in I_v(P)) \\
& \mbox{ iff } \overline{\mathfrak{M}} \vDash_w^\sigma \boxf P(x_1, \ldots, x_n) \\
& \mbox{ iff } \overline{\mathfrak{M}} \vDash_w^\sigma P(x_1, \ldots, x_n)^t
\end{align*}

The cases where $A=\bot$, $A=B \land C$, and $A=B \lor C$ are straightforward.

If $A=B \to C$, then
using the inductive hypothesis, we have
\begin{align*}
\mathfrak{M} \vDash_w^\sigma B \to C & \mbox{ iff } (\forall v\in W)(w R v \Rightarrow (\mathfrak{M} \vDash_v^\sigma B \Rightarrow \mathfrak{M} \vDash_v^\sigma C)) \\
& \mbox{ iff } (\forall v\in W)(w R v \Rightarrow (\overline{\mathfrak{M}} \vDash_v^\sigma B^t \Rightarrow \overline{\mathfrak{M}} \vDash_v^\sigma C^t)) \\
& \mbox{ iff } \overline{\mathfrak{M}} \vDash_{w}^\sigma \boxf(B^t \to C^t) \\
& \mbox{ iff } \overline{\mathfrak{M}} \vDash_{w}^\sigma (B \to C)^t
\end{align*}

If $A=\forall x B$, then
using the inductive hypothesis, we have
\begin{align*}
\mathfrak{M} \vDash_w^\sigma \forall x B & \mbox{ iff } (\forall v\in W)(w R v \Rightarrow \mbox{for each $v$-assignment $\tau$ that is} \\
& \hphantom{ iff } \mbox{an $x$-variant of $\sigma$ we have } \mathfrak{M} \vDash_v^\tau B)\\
& \mbox{ iff } (\forall v\in W)(w R v \Rightarrow \mbox{for each assignment $\tau$ that is} \\
& \hphantom{ iff } \mbox{an $x$-variant of $\sigma$ with $\tau(x) \in D_v$ we have } \overline{\mathfrak{M}} \vDash_v^\tau B^t) \\
& \mbox{ iff } \overline{\mathfrak{M}} \vDash_{w}^{\sigma} \boxf \forall x B^t\\
& \mbox{ iff } \overline{\mathfrak{M}} \vDash_{w}^{\sigma} (\forall x B)^t
\end{align*}

If $A=\exists x B$, then
using the inductive hypothesis, reflexivity of $R$, Lemma~\ref{lem:A^t truth increasing}, and the fact that $v R w$ implies $D_v \subseteq D_w$, we have
\begin{align*}
\mathfrak{M} \vDash_w^\sigma \exists x B & \mbox{ iff } \mbox{there is a $w$-assignment $\tau$ that is an $x$-variant of $\sigma$}\\
& \hphantom{ iff } \mbox{such that } \mathfrak{M} \vDash_w^\tau  B \\
& \mbox{ iff } \mbox{there is an assignment $\tau$ that is an $x$-variant of $\sigma$} \\
& \hphantom{ iff } \mbox{with } \tau(x) \in D_w \mbox{ such that } \overline{\mathfrak{M}} \vDash_w^{\tau}  B^t \\
& \mbox{ iff } \mbox{there is $v\in W$ such that $v R w$ and an assignment $\rho$ that is} \\
& \hphantom{ iff } \mbox{an $x$-variant of $\sigma$ with } \rho(x) \in D_v \mbox{ such that } \overline{\mathfrak{M}} \vDash_v^{\rho}  B^t \\
& \mbox{ iff } \overline{\mathfrak{M}} \vDash_{w}^{\sigma} \diap \exists x B^t\\
& \mbox{ iff } \overline{\mathfrak{M}} \vDash_{w}^{\sigma} (\exists x B)^t
\end{align*}

(ii). By Definition~\ref{def:truth and validity in iqc}, $\mathfrak{M} \vDash_w A$ iff $\mathfrak{M} \vDash_w^\sigma A$ for each $w$-assignment $\sigma$. As noted in Remark~\ref{rem:mbar}, $w$-assignments in $\mathfrak{F}$ are exactly the $w$-inner assignments in $\overline{\mathfrak{F}}$. Therefore, by (i), $\mathfrak{M} \vDash_w A$ iff $\overline{\mathfrak{M}} \vDash_w^\sigma A^t$ for each $w$-inner assignment $\sigma$. It follows from the interpretation of the universal quantifier in $\overline{\mathfrak{M}}$
that $\overline{\mathfrak{M}} \vDash_{w}^\sigma A^t$ for each $w$-inner assignment $\sigma$ iff $\overline{\mathfrak{M}} \vDash_{w}  \forall x_1 \cdots \forall x_n A^t$.
Thus, $\mathfrak{M} \vDash_w A$ iff $\overline{\mathfrak{M}} \vDash_{w}  \forall x_1 \cdots \forall x_n A^t$.
\end{proof}

\begin{theorem}\label{thm: full formula}
Let $A$ be a formula of $\mathcal{L}$ and $x_1, \ldots, x_n$ the free variables of $A$. If $\qost \vdash \forall x_1 \cdots \forall x_n A^t$, then $\iqc \vdash A$.
\end{theorem}

\begin{proof}
Suppose $\iqc \nvdash A$. Theorem~\ref{thm:sound and compl of iqc} implies that there is an $\iqc$-model $\mathfrak{M}$ such that $\mathfrak{M} \nvDash_w A$ for some world $w$. By Proposition~\ref{prop:validity and mbar}(ii), $\overline{\mathfrak{M}} \nvDash_w \forall x_1 \cdots \forall x_n A^t$. Thus, $\qost \nvdash \forall x_1 \cdots \forall x_n A^t$ by Theorem~\ref{prop:soundness qost}.
\end{proof}

By putting Theorems~\ref{thm: faith formula} and~\ref{thm: full formula} together we arrive at the main result of the paper mentioned in the introduction.

\begin{theorem}
\begin{itemize}
\item[]
\item Let $A$ be a formula of $\mathcal{L}$ and $x_1, \ldots, x_n$ the free variables of $A$. We have
\[
\iqc \vdash A \mbox{ iff } \qost \vdash \forall x_1 \cdots \forall x_n A^t.
\]
\item If $A$ is a sentence of $\mathcal{L}$, then
\[
\iqc \vdash A \mbox{ iff } \qost \vdash A^t.
\]
\end{itemize}
\end{theorem}

\begin{remark}\label{rem:adding constants}
If we allow constants in $\mathcal{L}$, Theorem~\ref{thm: full formula} is no longer true in its current form. Indeed, constants in $\iqc$ and $\qost$ behave like free variables and we would have the problem described in Remark~\ref{rem:need univ closure}. However, it can be adjusted as follows. Let $A$ be a formula containing free variables $x_1, \ldots, x_n$ and constants $c_1, \ldots c_m$. If $A(y_1/c_1, \ldots, y_m/c_m)$ is the formula obtained by replacing all the constants with fresh variables $y_1, \ldots, y_m$,
then $\iqc \vdash A$ iff $\qost \vdash \forall x_1 \cdots \forall x_n \forall y_1 \cdots \forall y_m A^t(y_1/c_1, \ldots, y_m/c_m)$.
\end{remark}

\section{Open problems}\label{sec:open problems}

As follows from Theorem~\ref{prop:soundness qost}, $\qost$ is sound with respect to the class of $\qost$-frames. However, its completeness remains an interesting open problem. The standard Henkin construction was modified by Hughes and Cresswell~\cite{CH96} and Corsi~\cite{Cor02} to obtain completeness of $\qok$.
If we adapt their technique to $\qost$, we obtain two relations $R_F$ and $R_P$ on the canonical model, one coming from
$\boxf$ and the other from $\boxp$.
There does not seem to be an obvious way to select an appropriate submodel in which the restrictions of these two relations are inverses of each other because the outer domains of accessible worlds are forced to increase by the construction. This problem disappears when constructing the canonical model for $\qst$
because the presence of $\bff$ and $\cbfp$ in each world
allows us to select witnesses without expanding the domains of accessible worlds, thus yielding that $\qst$ is sound and complete with respect to the class of $\qst$-frames.

The problem of completeness of $\qost$ seems to be closely related to the open problem, stated in~\cite[p.~1510]{Cor02}, of whether $\qok+\barc$ is Kripke complete. It appears that answering one of these problems could also provide an answer to the other.

One of the reviewers pointed out that another natural direction is to study the intermediate predicate logics and the corresponding extensions of $\qost$ for which our temporal translation remains full and faithful. Finally, it is worth investigating whether other tense predicate logics (such as the ones considered in~\cite{GM88})
could be used for translating $\iqc$ fully and faithfully. Some such systems admit presheaf semantics which is more general than Kripke semantics.

\appendix \section{Additional facts needed in Sections~\ref{sec:qost} and~\ref{sec:translation}}\label{sec:appendix a}

\begin{proposition}\label{prop:bfp in qost}
$\qost \vdash \bfp$.
\end{proposition}

\begin{proof}
We first show that $\qost \vdash \diaf \forall x B \to \forall x \diaf B$ for any formula $B$. We have the proof\\

\begin{tabular}{l l}
1. \quad & $\forall x (\forall x B \to B)$ \\
2. & $\forall x \boxf (\forall x B \to B)$ \\
3. & $\boxf (\forall x B \to B) \to (\diaf \forall x B \to \diaf B)$ \\
4. & $\forall x \boxf (\forall x B \to B) \to \forall x (\diaf \forall x B \to \diaf B)$ \\
5. & $\forall x (\diaf \forall x B \to \diaf B)$ \\
6. & $\forall x \diaf \forall x B \to \forall x \diaf B$ \\
7. & $\diaf \forall x B \to \forall x \diaf B$\\
\end{tabular}

\medskip
Here 1 is an instance of $\uio$;
2 is obtained from 1 by adding
$\boxf$ inside $\forall x$ by applying (N\textsubscript{F}), $\cbff$, and (MP);
3 is a substitution instance of the $K$-theorem $\boxf (C \to D) \to (\diaf C \to \diaf D)$ for $\boxf$;
4 is obtained from 3 by first adding and then distributing $\forall x$
inside the implication by applying (Gen), axiom (ii) of $\qok$, and (MP);
5 follows from 2 and 4 by (MP);
6 is obtained from 5 by distributing $\forall x$;
and 7 follows from 6 and axiom (iv) of $\qok$.\\

We now prove $\forall x \boxp A \to \boxp \forall x A$.\\

\begin{tabular}{l l}
1. \quad & $\forall x \boxp A \to \boxp \diaf \forall x \boxp A$ \\
2. & $\diaf \forall x \boxp A \to \forall x \diaf \boxp A$ \\
3. & $\boxp \diaf \forall x \boxp A \to \boxp  \forall x \diaf \boxp A$ \\
4. & $\diaf \boxp A \to A$ \\
5. & $\forall x \diaf \boxp A \to \forall x A$ \\
6. & $\boxp \forall x \diaf \boxp A \to \boxp \forall x A$ \\
7. & $\forall x \boxp A \to \boxp \forall x A$\\
\end{tabular}

\medskip
Here 1 is an instance of axiom (i) of $\st$; 2 is an instance of $\diaf \forall x B \to \forall x \diaf B$ proved above;
3 and 6 follow from 2 and 5 by adding and distributing $\boxp$ in the implication;
 4 is an instance of the $\st$-theorem $\diaf \boxp C \to C$;
 5 is obtained from 4 by adding and distributing $\forall x$;
  and 7 follows from 1, 3, and  6.
\end{proof}

\begin{lemma} \label{lem:diap At to At}
If $A$ is a formula of $\mathcal{L}$, then $\qost \vdash A^t \to \boxf A^t$ and $\qost \vdash \diap A^t \to A^t$.
\end{lemma}

\begin{proof}
We only prove that $\qost \vdash A^t \to \boxf A^t$ since it implies that $\qost \vdash \diap A^t \to A^t$. The proof is by induction on the complexity of $A$. If $A=\bot$, then $A^t=\bot$ and it is clear that $\qost \vdash \bot \to \boxf \bot$.

If $A$ is either an atomic formula $P(x_1, \ldots, x_n)$ or of the form $B \to C$ or $\forall x B$, then $A^t$ is of the form $\boxf D$. Therefore, the $4$-axiom $\boxf D \to \boxf  \boxf D$
implies that in all these cases $\qost \vdash A^t \to \boxf A^t$.

If $A=\exists x B$, then $A^t=\diap \exists x B^t$. So $\boxf A^t= \boxf \diap \exists x B^t$ and $\qost \vdash \diap \exists x B^t \to \boxf \diap \exists x B^t$ because it is a substitution instance of
 the $\st$-theorem $\diap C \to \boxf \diap C$.
Finally, if $A=B \land C$ or $A=B \lor C$, then we have $A^t=B^t \land C^t$ or $A^t=B^t \lor C^t$. By inductive hypothesis, $\qost \vdash  B^t \to \boxf B^t$ and $\qost \vdash C^t \to \boxf C^t$. Since $\qost \vdash (\boxf B^t \land \boxf C^t) \to \boxf (B^t \land C^t)$ and $\qost \vdash (\boxf B^t \lor \boxf C^t) \to \boxf (B^t \lor  C^t)$,
we obtain $\qost \vdash (B^t \land C^t) \to \boxf(B^t \land C^t)$ and $\qost \vdash (B^t \lor C^t) \to \boxf(B^t \lor C^t)$.
\end{proof}

\begin{lemma} \label{lem:useful facts qost}
The following are theorems of $\qost$:
\begin{enumerate}
\item $\forall y (A(y/x) \to \exists x A)$.
\item $\forall x (A \to B) \to (A \to \forall x B)$ if $x$ is not free in $A$.
\item $\forall x (A \to B) \to (\exists x A \to B)$ if $x$ is not free in $B$.
\end{enumerate}
\end{lemma}

\begin{proof}
Follows from \cite[Lem.~1.3]{Cor02}.
\end{proof}

\begin{lemma}\label{lem:translation of axioms}
For formulas $A,B$ of $\mathcal{L}$, the following are theorems of $\qost$.
\begin{enumerate}
\item $\boxf (\boxf \forall x A^t \to A^t)$ if $x$ is not free in $A$.
\item $\forall y \boxf (\boxf \forall x A^t \to A(y/x)^t)$.
\item $\boxf (A^t \to \diap \exists x A^t)$ if $x$ is not free in $A$.
\item $\forall y \boxf (A(y/x)^t \to \diap \exists x A^t)$.
\item $\boxf(\boxf \forall x \boxf(A^t \to B^t) \to \boxf(A^t \to \boxf \forall x B^t))$ if $x$ is not free in $A$.
\item $\boxf ( \boxf \forall x \boxf (A^t\to B^t) \to \boxf(\diap \exists x A^t \to B^t ))$ if $x$ is not free in $B$.
\end{enumerate}
\end{lemma}

\begin{proof}
Note that $x$ is free in $A$ iff it is free in $A^t$, and $A(y/x)^t=A^t(y/x)$. \\

(i). We have the proof\\

\begin{tabular}{l l}
1. \quad & $\forall x A^t \to A^t$ \\
2. & $\boxf \forall x A^t \to A^t$ \\
3. & $\boxf (\boxf \forall x A^t \to A^t)$\\
\end{tabular}

\medskip
where 1 is an instance of $\nid$ because $x$ is not free in $A^t$;
2 is obtained from 1 by applying
the T-axiom for $\boxf$; 3 is obtained from 2 by (N\textsubscript{F}).\\

(ii). We have the proof\\

\begin{tabular}{l l}
1. \quad & $\forall y (\forall x A^t \to A^t(y/x))$ \\
2. & $\forall y (\boxf \forall x A^t \to A^t(y/x))$ \\
3. & $\forall y \boxf (\boxf \forall x A^t \to A^t(y/x))$\\
\end{tabular}

\medskip
where 1 is an instance of $\uio$;
2 follows from 1 by applying
the T-axiom for $\boxf$ inside $\forall y$;
3 is obtained from 2 by introducing $\boxf$ inside $\forall y$.\\

(iii). We have the proof\\

\begin{tabular}{l l}
1. \quad & $A^t \to \exists x A^t$ \\
2. & $A^t \to \diap \exists x A^t$ \\
3. & $\boxf (A^t \to \diap \exists x A^t)$\\
\end{tabular}

\medskip
where 1 is an instance of $C \to \exists x C$, with $x$ not free in $C$, which is equivalent to $\nid$; 2 follows from 1 by the T-axiom for $\diap$; 3 is obtained from~2 by (N\textsubscript{F}).\\

(iv). We have the proof\\

\begin{tabular}{l l}
1. \quad & $\forall y (A^t(y/x) \to \exists x A^t)$ \\
2. & $\forall y (A^t(y/x) \to \diap \exists x A^t)$ \\
3. & $\forall y \boxf (A^t(y/x) \to \diap \exists x A^t)$\\
\end{tabular}

\medskip
where 1 follows from
Lemma~\ref{lem:useful facts qost}(i);
2 follows from 1 by applying the T-axiom for $\diap$
inside $\forall y$;
3 is obtained from 2 by introducing $\boxf$ inside $\forall y$.\\

(v). We have the proof\\

\begin{tabular}{l l}
1. \quad & $\forall x (A^t \to B^t) \to (A^t \to \forall x B^t)$ \\
2. & $\forall x \boxf(A^t \to B^t) \to (A^t \to \forall x B^t)$ \\
3. & $\boxf \forall x \boxf(A^t \to B^t) \to (\boxf A^t \to \boxf \forall x B^t)$ \\
4. & $\boxf \forall x \boxf(A^t \to B^t) \to (A^t \to \boxf \forall x B^t)$ \\
5. & $\boxf \forall x \boxf(A^t \to B^t) \to \boxf(A^t \to \boxf \forall x B^t)$\\
6. & $\boxf(\boxf \forall x \boxf(A^t \to B^t) \to \boxf(A^t \to \boxf \forall x B^t))$\\
\end{tabular}

\medskip
where 1 follows from
Lemma~\ref{lem:useful facts qost}(ii);
2 follows from 1 by applying the T-axiom for $\boxf$;
3 is obtained from 2 by adding and distributing $\boxf$;
4 follows from 3
by Lemma~\ref{lem:diap At to At};
5 is obtained from 4 by adding and distributing $\boxf$ and getting rid of one $\boxf$ in the antecedent using the $4$-axiom;
6 follows from 5 by (N\textsubscript{F}).\\

(vi). We have the proof\\

\begin{tabular}{l l}
1. \quad & $\forall x (A^t \to B^t) \to (\exists x A^t \to B^t)$ \\
2. & $\forall x (A^t \to B^t) \to (\exists x \diap A^t \to B^t)$ \\
3. & $\forall x \boxf (A^t \to B^t) \to (\exists x \diap A^t \to B^t)$ \\
4. & $\forall x \boxf (A^t \to B^t) \to (\diap \exists x A^t \to B^t)$\\
5. & $\boxf \forall x \boxf (A^t \to B^t) \to \boxf(\diap \exists x A^t \to B^t)$\\
6. & $\boxf (\boxf \forall x \boxf (A^t \to B^t) \to \boxf(\diap \exists x A^t \to B^t))$\\
\end{tabular}

\medskip
where 1 follows from
Lemma~\ref{lem:useful facts qost}(iii); 2 follows from 1 by
Lemma~\ref{lem:diap At to At};
3 follows from 2 by applying
the T-axiom for $\boxf$;
4 follows from 3 and
the fact that $\qost \vdash \diap \exists x A^t \to \exists x \diap A^t$
because it is a consequence of
$\bfp$; 5 is obtained from 4 by adding and distributing $\boxf$;
6 follows from 5 by (N\textsubscript{F}).
\end{proof}

\begin{lemma}\label{lem:faithful axiom}
If $C$ is an instance of an axiom scheme of $\iqc$ and $\mathbf{x}$ is the list of free variables in $C$, then $\qost \vdash \forall \mathbf{x} \, C^t$.
\end{lemma}

\begin{proof}
If $C$ is an instance of a theorem of $\ipc$, then it follows from the faithfulness of the G\"odel translation in the propositional case that $C^t$ is a theorem of $\qost$ (since $\boxf$ is an $\sfour$-modality).
Applying (Gen) to each free variable of $C^t$ then yields a proof of $\forall \mathbf{x} \, C^t$ in $\qost$. Translations of the axiom schemes of Definition~\ref{def:iqc} give:
\[
\begin{array}{l l}
 (\forall x A \to A (y/x))^t & =\boxf (\boxf \forall x A^t \to A(y/x)^t) \\[1.5ex]
(A(y/x) \to \exists x A)^t & =\boxf (A(y/x)^t \to \diap \exists x A^t) \\[1.5ex]
 \multicolumn{2}{l}{(\forall x (A \to B) \to (A \to \forall x B))^t} \\[1ex]
 & =\boxf(\boxf \forall x \boxf(A^t \to B^t) \to \boxf(A^t \to \boxf \forall x B^t)) \\[1.5ex]
 \multicolumn{2}{l}{(\forall x (A \to B) \to  (\exists x A \to B))^t} \\[1ex]
 & = \boxf ( \boxf \forall x \boxf (A^t\to B^t) \to \boxf(\diap \exists x A^t \to B^t ))
\end{array}
\]
If $C$ is an instance of one of these axiom schemes, then we obtain a proof of $\forall \mathbf{x} \, C^t$ in $\qost$ by Lemma~\ref{lem:translation of axioms} and by applying (Gen) to the free variables of $C$.
More precisely, for the first axiom we use (i) of Lemma~\ref{lem:translation of axioms} when $x$ is not free in $A$ and (ii) when $x$ is free in $A$.
Similarly, for the second axiom we use (iii) or (iv) of Lemma~\ref{lem:translation of axioms}. Finally, for the third axiom we use (v) and for the fourth axiom we use (vi) of Lemma~\ref{lem:translation of axioms}.
\end{proof}

\begin{lemma}\label{lem:faithful MP}
Let $A,B$ be formulas of $\mathcal{L}$, $\mathbf{x}$ the list of variables free in ${A \to B}$, $\mathbf{y}$ the list of variables free in $A$, and $\mathbf{z}$ the list of variables free in $B$. If $\qost \vdash \forall \mathbf{x} (A \to B)^t$ and $\qost \vdash \forall \mathbf{y} A^t$, then $\qost \vdash \forall \mathbf{z} B^t$.
\end{lemma}

\begin{proof}
Let $\mathbf{u}$ be the list of variables free in $A$ but not in $B$, $\mathbf{v}$ the list of variables free in $B$ but not in $A$, and $\mathbf{w}$ the list of variables free in both $A$ and $B$. We then have that $\mathbf{x}$ is the union of $\mathbf{u}$, $\mathbf{v}$, and $\mathbf{w}$; $\mathbf{y}$ is the union of $\mathbf{u}$ and $\mathbf{w}$; and $\mathbf{z}$ is the union of $\mathbf{v}$ and $\mathbf{w}$. Thus, we want to show that if $\qost \vdash \forall \mathbf{u} \, \forall \mathbf{v} \, \forall \mathbf{w} (A \to B)^t$ and $\qost \vdash \forall \mathbf{u} \, \forall \mathbf{w} A^t$, then $\qost \vdash \forall \mathbf{v} \, \forall \mathbf{w} B^t$.
We have the proof\\

\begin{tabular}{l l}
1. \quad & $\forall \mathbf{u} \, \forall \mathbf{v} \, \forall \mathbf{w} \, \boxf (A^t \to B^t)$ \\
2. & $\forall \mathbf{u} \, \forall \mathbf{w} \, \forall \mathbf{v} \, \boxf (A^t \to B^t)$ \\
3. & $\forall \mathbf{u} \, \forall \mathbf{w} \, \forall \mathbf{v} \, (\boxf A^t \to \boxf B^t)$ \\
4. & $\forall \mathbf{u} \, \forall \mathbf{w} \, (\boxf A^t \to \forall \mathbf{v} \, \boxf B^t)$ \\
5. & $\forall \mathbf{u} \, \forall \mathbf{w} \, \boxf A^t \to \forall \mathbf{u} \, \forall \mathbf{w} \,  \forall \mathbf{v} \, \boxf B^t$ \\
6. & $\forall \mathbf{u} \, \forall \mathbf{w} A^t$ \\
7. & $\forall \mathbf{u} \, \forall \mathbf{w} \, \boxf A^t$ \\
8. & $\forall \mathbf{u} \, \forall \mathbf{w} \, \forall \mathbf{v} \, \boxf B^t$ \\
9. & $\forall \mathbf{u} \, \forall \mathbf{w} \, \forall \mathbf{v} \, B^t$ \\
10 & $\forall \mathbf{w} \, \forall \mathbf{v} \, B^t$ \\
11 & $\forall \mathbf{v} \, \forall \mathbf{w} \, B^t$\\
\end{tabular}

\medskip
where 1 and 6 are assumptions;
2 and 11 follow from 1 and 10 by switching the order of quantification;
3 is obtained from 2 by distributing $\boxf$ inside the universal quantifiers;
4 follows from
Lemma~\ref{lem:useful facts qost}(ii) because all the variables in $\mathbf{v}$ are not free in $\boxf A^t$; 5 is obtained by distributing the universal quantifiers; 7 follows from 6 by introducing $\boxf$ inside the quantifiers;
8 is obtained by (MP) from 5 and 7; 9 follows from 8 by
the T-axiom for $\boxf$;
10 follows from 9 by $\nid$ because no variable in $\mathbf{u}$ is free in $B^t$.
\end{proof}

\begin{lemma}\label{lem:faithful gen}
Let $A$ be a formula of $\mathcal{L}$, $x$ a variable, $\mathbf{y}$ the list of variables free in $A$, and $\mathbf{z}$ the list of variables free in $\forall x A$. If $\qost \vdash \forall \mathbf{y} A^t$, then $\qost \vdash \forall \mathbf{z} \, (\forall x A)^t$.
\end{lemma}

\begin{proof}
If $x$ is in $\mathbf{y}$, then
without loss of generality we may assume that $\mathbf{y}$ is $\mathbf{z}$ concatenated with $x$. Therefore,
by assumption we have $\qost \vdash \forall \mathbf{z} \, \forall x A^t$.
If $x$ is not in $\mathbf{y}$,
 then $\mathbf{y}=\mathbf{z}$.
Thus, by (Gen) for $x$ and by switching the order of quantifiers, we again obtain
$\qost \vdash \forall \mathbf{z} \, \forall x A^t$.
We can then introduce $\boxf$ inside the quantifiers to obtain
   $\qost \vdash \forall \mathbf{z} \, \boxf \forall x A^t$ which means $\qost \vdash \forall \mathbf{z} \, (\forall x A)^t$.
\end{proof}

%% Bibliography
%% Make sure to use the bibliographystyle aiml20.

\end{document}